\documentclass{svmult}

\usepackage{graphicx}    % standard LaTeX graphics tool
\usepackage{color}
\usepackage{amsmath}
\usepackage{amssymb}
\usepackage[english]{babel}
\usepackage{times}

\newcommand{\R}{\mathbf{R}}
\newcommand{\set}[1]{\left\{#1\right\}}
\newcommand{\norm}[1]{\Vert#1\Vert}
\newcommand{\eoproof}{\hfill$\square$}

\makeindex             % used for the subject index
\begin{document}

% Title of the paper and the authors.
\title*{Real-Time Sequential Convex Programming for Optimal Control Applications}
\titlerunning{Real-Time Sequential Convex Programming}
% Use \titlerunning{Short Title} for an abbreviated version of
% your contribution title if the original one is too long
\author{Tran Dinh Quoc\inst{\dagger}, Carlo Savorgnan\inst{\dagger} \and
Moritz Diehl\inst{\dagger}}
% Use \authorrunning{Short Title} for an abbreviated version of
% your contribution title if the original one is too long
\institute{$^\dagger$ Department of Electrical Engineering (ESAT-SCD) and Optimization in Engineering Center (OPTEC), K.U. Leuven, 
Kasteelpark Arenberg 10, B-3001 Leuven, Belgium \textrm{\{quoc.trandinh, carlo.savorgnan, moritz.diehl\}@esat.kuleuven.be}}
%
% Use the package "url.sty" to avoid
% problems with special characters
% used in your e-mail or web address
%
\maketitle

%+ Abstract.
\begin{abstract}
This paper proposes real-time sequential convex programming (RTSCP), a method for solving a sequence of nonlinear optimization problems
depending on an online parameter.
We provide a contraction estimate for the proposed method and, as a byproduct, a new proof of the local convergence of sequential convex
programming.
The approach is illustrated by an example where RTSCP is applied to nonlinear model predictive control.
\end{abstract}

%%%%%%%%%%%%%%%%%%%%%%%%%%%%%%%%%%%%%%%%%%%%%%%%%%%%%%%%%%%%%%%%%%%%%%%%%%%%%%%%%%%%%%%%%%%%%%%%%%%%%%%%%%%%%%
%+1. Introduction
\section{Introduction and motivation}\label{se:intro}
\vskip -0.3cm
Consider a parametric optimization problem of the form:
\begin{equation}\label{eq:param_prob}
\makeatletter
\def\tagform@#1{\maketag@@@{#1\@@italiccorr}}
\makeatother
\left\{\begin{array}{cl}
\displaystyle\min_x & c^T x \\
\textrm{s.t.}  & g(x) + M\xi  = 0,\; x\in\Omega,
\end{array}\right.
\tag{$\textrm{P}(\xi)$}
\end{equation}
where $x,c\in\R^n$, $g:\R^n\to\R^m$ is a nonlinear function, $\Omega\subseteq\R^n$ is a convex set, the parameter 
$\xi$ belongs to a given set $\Gamma\subseteq\R^p$, and $M\in\R^{m\times p}$ is a given matrix.

This paper deals with the efficient calculation of approximate solutions to a sequence of problems of the form \ref{eq:param_prob} where 
the parameter $\xi$ is varying slowly. In other words, for a sequence $\{\xi_k\}_{k\geq 1}$ such that $\norm{M(\xi_{k+1}-\xi_k)}$ is small,
we want to solve problem $\textrm{P}(\xi_k)$ in an efficient way without requiring too much accuracy in the result.

In practice, sequences of problems of the form \ref{eq:param_prob} can be solved in the framework of nonlinear model predictive control
(MPC). MPC is an optimal control technique which avoids computing an optimal control law in a feedback form, which is often a numerically
intractable problem. 
A popular way of solving the optimization problem to calculate the control sequence is using either interior point methods 
\cite{Biegler2000} or sequential quadratic programming (SQP) \cite{Biegler1991,Bock2000,Helbig1998}. A drawback of using SQP is that this
method may require several iterations before convergence and therefore  the computation time may be too large for a real-time
implementation.
A solution to this problem was proposed in \cite{Diehl2002b}, where the real-time iteration (RTI) technique was introduced. Extensions to
the original idea and some theoretical results are reported in \cite{Diehl2005,Diehl2007b,Diehl2005b}. Similar nonlinear MPC algorithms are
proposed in~\cite{Ohtsuka2004,Zavala2009}.
RTI is based on the observation that for several practical applications of nonlinear MPC, the data of two successive optimization problems 
to be solved in the MPC loop is numerically close.
In particular, if we express these optimization problems in the form \ref{eq:param_prob}, the parameter $\xi$ usually represents the 
current state of the system, which, for most applications, doesn't change significantly in two successive measurements.
The RTI technique consists of performing only the first step of the usual SQP algorithm which is initialized using the solution calculated 
in the previous MPC iteration.

\vskip 0.1cm
\noindent\textbf{Contribution.}
Before stating the main contributions of the paper we need to outline the (full-step) sequential convex programming (SCP) algorithm framework applied to
problem $\textrm{P}(\xi)$ for a given value $\xi_k$ of the parameter $\xi$:
\begin{enumerate}
\item Choose a starting point $x^0\in\Omega$ and set $j := 0$.
\item Solve the convex approximation of $\textrm{P}(\xi_k)$:
\begin{equation}\label{eq:SCP_method}
\makeatletter
\def\tagform@#1{\maketag@@@{#1\@@italiccorr}}
\makeatother
\left\{\begin{array}{cl}
\displaystyle
\min_x &c^T x \\
\textrm{s.t.} &g'(x^j)(x-x^j) + g(x^j) + M\xi_k = 0, \\
       &x \in \Omega
\end{array}\right.
\tag{$\textrm{P}_{\textrm{cvx}}(x^j;\xi_k)$}
\end{equation}
to obtain a solution $x^{j+1}$, where $g'(\cdot)$ is the Jacobian matrix of $g(\cdot)$.
\item If the stopping criterion is satisfied then: STOP. Otherwise, set $j := j+1$ and go back to Step 2.
\end{enumerate}
The real-time sequential convex programming (RTSCP) method proposed in this paper combines the RTI technique and the SCP algorithm: 
instead of solving with SCP every $\textrm{P}(\xi_k)$ to full accuracy, RTSCP solves only one convex approximation
$\textrm{P}_{\text{cvx}}(x^{k-1};\xi_k)$ using as a linearization point $x^{k-1}$, which is the approximate solution of
$\textrm{P}(\xi_{k-1})$ calculated at the previous iteration.
Therefore, RTSCP solves a sequence of convex problems corresponding to the different problems $\textrm{P}(\xi_k)$. This method is suitable
for the problems that contain a general convex substructure such as nonsmooth convex cost, second order or semidefinte cone constraints which
may not be convenient for SQP methods.

In this paper we provide a contraction estimate for RTSCP which can be interpreted in the following way: if the linearization of the first 
problem $\textrm{P}(\xi_0)$ is close enough to the solution of the problem and the quantity $\norm{M(\xi_{k+1}-\xi_k)}$ is not too big
(which is the case for many problems arising from nonlinear MPC), RTSCP provides a sequence of good approximations of the sequence of
optimal solutions of the problems $\textrm{P}(\xi_k)$. As a byproduct of this result, we obtain a new proof of local convergence for the SCP
algorithm. 

The paper is organized as follows. Section \ref{se:RTSCP} proposes a description of the RTSCP algorithm. Section \ref{se:estimates} proves
the contraction estimate for the RTSCP method. The last section shows an application of the RTSCP method to nonlinear MPC.

%%%%%%%%%%%%%%%%%%%%%%%%%%%%%%%%%%%%%%%%%%%%%%%%%%%%%%%%%%%%%%%%%%%%%%%%%
%% 2. The RTSCP algorithm
\section{The RTSCP method}\label{se:RTSCP}
\vskip -0.3cm
As mentioned in the  previous section, SCP solves a possibly nonconvex optimization problem by solving a sequence of convex
subproblems which approximate the original problem locally. In this section, we combine RTI and SCP to obtain the RTSCP method. The method
consists of the following steps:
%\noindent{\bf RTSCP Method}
\begin{itemize}
\item[]{\bf Initialization. }  Find an initial value $\xi_1\in\Gamma$, choose a starting point $x^0\in\Omega$ and compute
the information needed at the first iteration such as derivatives, dependent variables, \dots. Set $k:=1$.
\item[] {\bf Iteration. }
\begin{enumerate}
\item Solve $\textrm{P}_{\textrm{cvx}}(x^{k-1};\xi_{k})$ (see Section \ref{se:estimates}) to obtain a solution $x^k$.
\item Determine a new parameter $\xi_{k+1}\in\Gamma$, update (or recompute) the information needed for the next step. Set $k:=k+1$ and go back to
Step 1.
\end{enumerate} 
\end{itemize}
One of the main tasks of the RTSCP method is to solve the convex subproblem $\textrm{P}_{\textrm{cvx}}(x^{k-1};\xi_{k})$  at
each iteration. This work can be done by either implementing an optimization method which exploits the
problem structure or relying on one of the many efficient software tools available nowadays.

% Remark 2.1.
\begin{remark}\label{re:init_point}
In the RTSCP method, a starting point $x^0$ in $\Omega$ is required. It can be any point in $\Omega$. But as we will show later [Theorem 1], if we choose $x^0$ close to the true solution of $\textrm{P}(\xi_0)$ and $\norm{M(\xi_1-\xi_0)}$ is sufficiently small, then the solution $x^1$ of
$\textrm{P}_{\textrm{cvx}}(x^0,\xi_1)$ is still close to the true solution of $\textrm{P}(\xi_1)$. Therefore, in practice, problem $\textrm{P}(\xi_0)$ can be solved approximately to get a starting point $x^0$.
\end{remark}

%+ Remark 2.2.
\begin{remark}\label{re:SQP}
Problem $\textrm{P}(\xi)$ has a linear cost function. However, RTSCP can deal directly with the problems where the cost function $f(x)$ is convex. If the cost function is quadratic and $\Omega$ is a polyhedral set then the RTSCP method collapses to the
real-time iteration of a Gauss-Newton method (see, e.g. \cite{Diehl2002}). 
\end{remark} 

%+ Remark 2.3.
\begin{remark}\label{re:measure}
In MPC, the parameter $\xi$ is usually the value of the state variables of a dynamic system at the current time $t$. In this case, $\xi$ is measured at each sample time based on the real-world dynamic system (see example in Section \ref{se:example}).
\end{remark}  
   
%%%%%%%%%%%%%%%%%%%%%%%%%%%%%%%%%%%%%%%%%%%%%%%%%%%%%%%%%%%%%%%%%%%%%%%%%%%%%%%
%+ 3. Approximate solution of multi-parametric optimization problem
\section{RTSCP contraction estimate}\label{se:estimates}
\vskip -0.3cm
The KKT conditions of problem \ref{eq:param_prob} can be written as
\begin{equation}\label{eq:KKT}
\begin{cases}
0\in c + g'(x)^T\lambda + N_{\Omega}(x)\\
0 = g(x) + M\xi,
\end{cases}
\end{equation}
where $N_{\Omega}(x) := \left\{u\in\R^n\;|\; u^T(v-x) \geq 0, \forall v\in\Omega\right\}$ if $x\in\Omega$ and $N_{\Omega}(x):=\emptyset$ if
$x\notin \Omega$, is the normal cone of $\Omega$ at $x$, and $\lambda$ is a Lagrange multiplier associated with $g$. 
Note that the constraint $x\in\Omega$ is implicitly included in the first line of \eqref{eq:KKT}. 
A pair $\bar{z}(\xi):=(\bar{x}(\xi), \bar{\lambda}(\xi))$ satisfying \eqref{eq:KKT} is called a KKT point  and $\bar{x}(\xi)$ is called a
stationary point of \ref{eq:param_prob}. We denote by $\Lambda(\xi)$ the set of KKT points at $\xi$.
 
In the sequel, we use $z$ for a pair $(x,\lambda)$, $\bar{z}^k$ is a KKT point of \ref{eq:param_prob} at $\xi_k$ and $z^k$ is a KKT point 
of \ref{eq:subprob_x} (defined below) at $\xi_{k+1}$ for $k\geq 0$. The symbols $\norm{\cdot}$ and $\norm{\cdot}_F$ stand for the $L_2$-norm and the Frobenius
norm, respectively.
 
Now, let us define $\varphi(z;\xi) := \begin{pmatrix}c + g'(x)^T\lambda\\ g(x) + M\xi\end{pmatrix}$ and $K:= \Omega\times\mathbf{R}^m$, then
the KKT system \eqref{eq:KKT} can be expressed as a parametric \textit{generalized equation} \cite{Robinson1980}:
\begin{equation}\label{eq:gen_equa}
0 \in \varphi(z;\xi) + N_K(z), 
\end{equation}
where $N_K(z)$ is the normal cone of $K$ at $z$.

Let $x^k\in\Omega$ be a solution of $\textrm{P}_{\textrm{cvx}}(x^{k-1};\xi_{k})$ at the $k$-iteration of RTSCP. We consider the following
parametric convex subproblem at Step 1 of the RTSCP algorithm:
\begin{equation}\label{eq:subprob_x}
\makeatletter
\def\tagform@#1{\maketag@@@{#1\@@italiccorr}}
\makeatother
\left\{\begin{array}{cl}
\displaystyle\min_x &c^T x \\
 \textrm{s.t.}  &g'(x^k)(x-x^k) + g(x^k) + M\xi_{k+1} = 0,\\
       &x\in \Omega.
\end{array}\right.
\tag{$\textrm{P}_{\textrm{cvx}}(x^k;\xi_{k+1})$}
\end{equation}
\vskip -0.1cm
If we define $\hat\varphi(z;x^k,\xi_{k+1}) := \begin{pmatrix}c + g'(x^k)^T\lambda\\ g(x^k) + g'(x^k)(x-x^k) + M\xi_{k+1}\end{pmatrix}$  then
the KKT condition for $\textrm{P}_{\textrm{cvx}}(x^k,\xi_{k+1})$ can also be represented as a parametric generalized equation:
\begin{equation}\label{eq:gen_eq_sub}
0 \in \hat\varphi(z; x^k, \xi_{k+1}) + N_K(z), 
\end{equation}
where $\eta_k := (x^k,\xi_{k+1})$ plays a role of parameter.
Suppose that the Slater constraint qualification condition holds for problem \ref{eq:subprob_x}, i.e.:
 \begin{equation*}
\text{ri}(\Omega)\cap\set{x:  g(x^k)+ g'(x^k)(x-x^k) + M\xi_{k+1} = 0}\neq \emptyset,
\end{equation*}
where $\text{ri}(\Omega)$ is the set of the relative interior points of $\Omega$. Then by convexity of $\Omega$, a point $z^{k+1} =
(x^{k+1},\lambda^{k+1})$ is a KKT point of the subproblem \ref{eq:subprob_x} if and only if $x^{k+1}$ is a solution of
\ref{eq:subprob_x} with a corresponding multiplier $\lambda^{k+1}$.

For a given KKT point $\bar{z}^k \in \Lambda(\xi_k)$ of $\textrm{P}(\xi_k)$, we define a set-valued mapping:
\begin{equation}\label{eq:L_def}
L(z; \xi) := \hat\varphi(z; \bar{x}^k, \xi) + N_K(z), 
\end{equation}
and $L^{-1}(\delta;\xi) := \set{z\in\R^{n+m}: \delta \in L(z; \xi) }$ for $\delta\in\R^{n+m}$ is its inverse mapping. 
Note that $0\in L(z;\xi)$ is indeed the KKT condition of $\textrm{P}_{\textrm{cvx}}(\bar{x}^k;\xi)$.   
For each $k\geq 0$, we make the following assumptions: 
\begin{itemize}
\item[]{\bf (A1)} The set of the KKT points $\Lambda_0 := \Lambda(\xi_0)$ is nonempty.
\item[]{\bf (A2)} The function $g$ is twice continuously differentiable on its domain.
\item[]{\bf (A3)} There exist a neighborhood $\mathcal{N}_{0}\subset\R^{n+m}$ of the origin and a neighborhood $\mathcal{N}_{\bar{z}^k}$ of
$\bar{z}^k$ such that for each $\delta\in \mathcal{N}_{0}$, $\psi_{k}(\delta) := \mathcal{N}_{\bar{z}^k}\cap L^{-1}(\delta; \xi)$ is
single-valued and Lipschitz continuous on $\mathcal{N}_{0}$ with a Lipschitz constant $\gamma > 0$.
\item[]{\bf (A4)} There exists a constant $0 \leq \kappa < 1/\gamma$ such that $\norm{E_g(\bar{z}^k)}_F \leq \kappa$,
where $E_g(\bar{z}^k) := \sum_{i=1}^m\bar{\lambda}^k_i\nabla^2g_i(\bar{x}^k) $.
\end{itemize}
Assumptions (A1) and (A2) are standard in optimization, while Assumption (A3) is related to the \textit{strong regularity} concept 
introduced by Robinson \cite{Robinson1980} for the parametric generalized equations of the form \eqref{eq:gen_equa}. It is important to note
that the strong regularity assumption follows from the strong second order sufficient optimality in nonlinear programming when the constraint qualification
condition (LICQ) holds \cite{Robinson1980}~\textrm{[Theorem 4.1]}.
In this paper, instead of the generalized linear mapping $L_R(z; \xi) :=  \varphi(\bar{z}^k;\xi) + \varphi'(\bar{z}^k)(z-\bar{z}^k) +
N_K(z)$  used in \cite{Robinson1980} to define strong regularity, in Assumption (A3) we use a similar form $L(z;\xi) =
\varphi(\bar{z}^k;\xi) + D(\bar{z}^k)(z-\bar{z}^k) + N_K(z)$, where  
\begin{equation*}\label{eq: D_def}
\varphi'(\bar{z}^k) = \begin{bmatrix}E_g(\bar{z}^k) & g'(\bar{x}^k)^T\\ g'(\bar{x}^k) & 0\end{bmatrix}, ~\text{and}~~ D(\bar{z}^k) = \begin{bmatrix}0 & g'(\bar{x}^k)^T\\ g'(\bar{x}^k) & 0\end{bmatrix}.
\end{equation*}
These expressions are different from each other only at the left-top corner $E_g(\bar{z}^k)$, the Hessian of the Lagrange function. 
Assumption (A3) corresponds to the standard strong regularity assumption (in the sense of Robinson \cite{Robinson1980}) of the subproblem \ref{eq:subprob_x} at the point $\bar{z}^{k}$, a KKT point of \eqref{eq:gen_equa} at $\xi=\xi_k$.

Assumption (A4) implies that either the  function $g$ should be ``weakly nonlinear'' (small second derivatives) in a neighborhood of  a stationary point or the corresponding Langrage multipliers are sufficiently small in this neighborhood. The latter case occurs if the optimal value of \ref{eq:param_prob} depends only weakly on perturbations of the nonlinear constraint $g(x)+M\xi=0$. 

%+ Theorem 2.1
\begin{theorem}[Contraction Theorem]\label{th:contract}
Suppose that Assumptions {\bf(A1)}-{\bf (A4)} are satisfied. Then there exist neighborhoods $\mathcal{N}_{\tau}$ of $\xi_k$, $\mathcal{N}_{\rho}$ of $\bar{z}^k$ and a single-valued function $\bar{z}: \mathcal{N}_{\tau}\to \mathcal{N}_{\rho}$ such that for all $\xi_{k+1}\in \mathcal{N}_{\tau}$, $\bar{z}^{k+1}:=\bar{z}(\xi_{k+1})$ is the unique KKT point of $\textrm{P}(\xi_{k+1})$ in $\mathcal{N}_{\rho}$ with respect  to parameter $\xi_{k+1}$ (i.e. $\Lambda(\xi_{k+1})\neq\emptyset$). 
Moreover, for any $\xi_{k+1} \in \mathcal{N}_{\tau}$, $z^{k} \in \mathcal{N}_{\rho}$ we have
\begin{equation}\label{eq:contr_est}
\norm{z^{k+1} - \bar{z}^{k+1}} \leq \omega_k\norm{z^k - \bar{z}^k} + c_k\norm{M(\xi_{k+1}-\xi_k)},
\end{equation}
where $\omega_k  \in (0,1)$, $c_k>0$ are constant, and $z^{k+1}$ is a KKT point of \ref{eq:subprob_x}.
\end{theorem}

%+ Proof of Theorem 3.2.
\begin{proof}
\vskip -0.2cm
The proof is organized in two parts and step by step. The first part proves $\Lambda_k:=\Lambda(\xi_k)\neq\emptyset$ for all $k\geq 0$ by induction and estimates the norm $\norm{\bar{z}^{k+1}-\bar{z}^k}$. The second part proves the inequality \eqref{eq:contr_est}.
%+ Picture of the proof.
\vskip -0.5cm
\setlength{\unitlength}{1mm}
\begin{picture}(60, 40)\label{fig:proof}
  \linethickness{0.075mm}
  \put(20, 0){\vector(0, 1){33}}
  \put(21,32){\color{blue}$\bar{z}(\xi)$}
  \put(20, 0){\vector(1, 0){50}} 
  \put(68, 2){$\xi$}
  {\color{blue}
  \linethickness{0.3mm}
  \qbezier(25, 5)(35, 27)(70, 28)
  }
  \linethickness{0.075mm}
  \put(25,12.2){\line(2,3){5}}
  \multiput(25,0)(0,1){13}{\circle*{0.01}}
  \multiput(30,0)(0,1){20}{\circle*{0.01}}
  \put(45,27.5){\line(6,1){15}}
  \multiput(46,0)(0,1){28}{\circle*{0.01}}
  \multiput(61,0)(0,1){31}{\circle*{0.01}}
  \put(25,12){\circle*{1}}
  {\color{red}\put(25,5){\circle*{1}}}
  \put(30,20){\circle*{1}}
  {\color{red}\put(30,13){\circle*{1}}}
  {\color{red}\put(45,23.7){\circle*{1}}}
  \put(45,27.5){\circle*{1}}
  \put(60,30){\circle*{1}}
  {\color{red}\put(60,27.5){\circle*{1}}}
  \multiput(30,20)(1,0.5){16}{\circle*{0.1}} 
  \put(24,-3){{\scriptsize$\xi_0$}}
  \put(29,-3){{\scriptsize$\xi_1$}}
  \put(44,-3){{\scriptsize$\xi_k$}}
  \put(59,-3){{\scriptsize$\xi_{k+1}$}}
  \put(21, 12){$z^0$}
  \put(20, 5){$\color{blue}\bar{z}^0$}
  \put(44, 29){$z^k$}
  \put(45.3, 20.5){\color{blue}$\bar{z}^k$}
  \put(59, 31){$z^{k+1}$}
  \put(60.5, 24.3){$\color{blue}\bar{z}^{k+1}$}
  \put(18.5,-3){\scriptsize$0$}
  \put(42.0,25.5){\color{red}\makebox(0,0){{\tiny[1]}$\left\{\right.$}}
  \put(62.3,28.8){\color{red}\makebox(0,0){$\left\}\right.${\tiny[2]}}}
  \put(45,0){\color{red}\tiny$\overbrace{\rule{15mm}{0cm}}$}
  \put(51.5,2.2){\color{red}\tiny[4]}
  \put(45,23.5){\color{red}\line(1,0){15}}
  \put(58.4,25.3){\color{red}\makebox(0,0){{\tiny[3]}$\left\{\right.$}}
  \put(20,20){\color{red}\vector(0,1){11}}
  \put(20,31){\color{red}\vector(0,-1){11}}
  \put(19,31){\line(1,0){2}}	
  \put(19,20){\line(1,0){2}}
  \put(21,25){$\mathcal{N}_{\rho}$}
  \put(42,5){\color{red}\vector(1,0){21}}
  \put(63,5){\color{red}\vector(-1,0){21}}
  \put(42,0){\line(0,1){6}}
  \put(63,0){\line(0,1){6}}
  \put(52,5.8){$\mathcal{N}_{\tau}$}
  % These are the lables.
  \put(76,22){\scriptsize{\color{red}[1]}~:~$\norm{z^k-{\color{blue}\bar{z}^k}}$}  
  \put(76,18){\scriptsize{\color{red}[2]}~:~$\norm{z^{k+1}-{\color{blue}\bar{z}^{k+1}}}$}
  \put(76,14){\scriptsize{\color{red}[3]}~:~$\color{blue}\norm{\bar{z}^{k+1}-\bar{z}^{k}}$}
  \put(76,10){\scriptsize{\color{red}[4]}~:~$\norm{\xi_{k+1}-\xi_k}$}
\end{picture}
\vskip 0.5cm
\centerline{\footnotesize \textbf{Fig. 1}. The approximate sequence $\{z^k\}_k$ along the manifold $\bar{z}(\cdot)$ of the KKT points.}
\vskip 0.2cm

%+ Part 1. 
\noindent\textbf{Part 1: }    
For $k=0$, $\Lambda_0 \neq\emptyset$ by Assumption (A1). Suppose that $\Lambda_k\neq\emptyset$ for $k\geq 0$, we will show that $\Lambda_{k+1}\neq\emptyset$. We divide the proof into four steps.

\noindent\textit{\underline{Step 1.1}. } We first provide the following estimations. Take any $\bar{z}^k\in\Lambda_k$. We define
\begin{equation}\label{eq:r_k}
r_k(z; \xi) := \hat\varphi(z;\bar{x}^k, \xi_k) - \varphi(z; \xi).
\end{equation}
Since $\gamma\kappa < 1$ by (A4), we can choose $\varepsilon>0$ sufficiently small such that $\gamma\kappa + 5\sqrt{3}\gamma\varepsilon < 1$. By the choice of $\varepsilon$, we also have $c_0 := \kappa + \sqrt{3}\varepsilon \in (0,1/\gamma)$. 
Since $g$ is twice continuously differentiable, there exist neighborhoods $\mathcal{N}_{\tau}\subseteq \mathcal{N}_{\xi_k}$ of $\xi_k$ and $\mathcal{N}_{\rho}\subseteq\mathcal{N}_{\bar{z}^k}$ of a radius $\rho>0$ centered at $\bar{z}^k$ such that:
$r_k(z;\xi)\in\mathcal{N}_{0}$, $\norm{E_g(z)-E_g(\bar{z}^k)}_F\leq \varepsilon$,  $\norm{E_g(z)-E_g(z^k)}_F\leq \varepsilon$, $\norm{g'(x)-g'(\bar{x}^k)}_F\leq\varepsilon$ and $\norm{g'(x)-g'(x^k)}_F\leq\varepsilon$ for all $z\in\mathcal{N}_{\rho}$.

Next, we shrink the neighborhood $\mathcal{N}_{\tau}$ of $\xi_k$, if necessary, such that:
\begin{equation}\label{eq:proof_eq1}
\norm{M(\xi-\xi_k)} \leq \rho(1-c_0)/\gamma.
\end{equation}

\vskip -0.1cm
\noindent\textit{\underline{Step 1.2}.} For any $z, z'\in\mathcal{N}_{\rho}$, we now estimate $\norm{r_k(z;\xi)-r_k(z';\xi)}$. 
From \eqref{eq:r_k} we have
\begin{eqnarray}\label{eq:r_k_est}
r_k(z;\xi) - r_k(z';\xi) &&= \hat\varphi(z;\bar{x}^k,\xi_k)-\hat\varphi(z';\bar{x}^k,\xi_k) - \varphi(z;\xi) + \varphi(z';\xi)\nonumber\\
[-1.5ex]\\[-1.5ex]
&& = \int_0^1B(z_t;\bar{x}^k)(z'-z)dt\nonumber,   
\end{eqnarray}
where $z_t := z + t(z'-z) \in\mathcal{N}_{\rho}$ and
\begin{equation}\label{eq:B_mat}
B(z;\hat{x}) = \begin{bmatrix}E_g(z) & g'(z)^T-g'(\hat{x})^T \\ g'(x)-g'(\hat{x}) & 0\end{bmatrix}. 
\end{equation}
Using the estimations of $E_g$ and $g'$ at \textit{Step 1.1}, it follows from \eqref{eq:B_mat} that 
\begin{eqnarray}\label{eq:B_mat_est}
\norm{B(z_t;\bar{x}^k)} &&\leq \norm{E_g(\bar{z}^k)}_F + \left[\norm{E_g(z_t)-E_g(\bar{z}^k)}_F^2 + 2\norm{g'(x_t)-g'(\bar{z}^k)}_F^2\right]^{1/2} \nonumber\\
[-1.5ex]\\[-1.5ex]
&& \leq \kappa + \sqrt{3}\varepsilon \equiv c_0.\nonumber 
\end{eqnarray}
Substituting \eqref{eq:B_mat_est} into \eqref{eq:r_k_est}, we get
\begin{equation}\label{eq:r_k_est2}
\norm{r_k(z;\xi)-r_k(z';\xi)} \leq c_0\norm{z-z'}. 
\end{equation}

\vskip -0.15cm
\noindent\textit{\underline{Step 1.3}.} Let us define $\Phi_{\xi}(z): = \mathcal{N}_{\bar{z}^k}\cap L(r_k(z;\xi);\xi_k)$. Next, we show that $\Phi_{\xi}(\cdot)$ is a contraction self-mapping onto $\mathcal{N}_{\rho}$ and then show that $\Lambda_{k+1} \neq\emptyset$.
 
Indeed, since $r_k(z;\xi)\in\mathcal{N}_{0}$, applying (A3) and \eqref{eq:r_k_est2}, for any $z, z'\in\mathcal{N}_{\rho}$, one has
\begin{equation}\label{eq:contract_map}
\norm{\Phi_{\xi}(z)-\Phi_{\xi}(z')} \leq \gamma\norm{r_k(z;\xi)-r_k(z';\xi)} \leq \gamma c_0\norm{z-z'}. 
\end{equation}
Since $\gamma c_0 \in (0,1)$ (see \textit{Step 1.1}), we conclude that $\Phi_{\xi}(\cdot)$ is a contraction mapping on $\mathcal{N}_{\rho}$. 
Moreover, since $\bar{z}^k = \mathcal{N}_{\bar{z}^k}\cap L^{-1}(0;\xi_k)$, it follows from (A3) and \eqref{eq:proof_eq1} that
\begin{equation*}
\norm{\Phi_{\xi}(\bar{z}^k)-\bar{z}^k} \leq \gamma\norm{r_k(\bar{z}^k;\xi)} = \gamma\norm{M(\xi-\xi_k)} \leq (1-\gamma c_0)\rho. 
\end{equation*}
Combining the last inequality, \eqref{eq:contract_map} and noting that $\norm{z-\bar{z}^k}\leq \rho$ we obtain
\begin{equation*}
\norm{\Phi_{\xi}(z)-\bar{z}^k} \leq \norm{\Phi_{\xi}(z)-\Phi_{\xi}(\bar{z}^k)} + \norm{\Phi_{\xi}(\bar{z}^k)-\bar{z}^k} \leq \rho,  
\end{equation*}
which proves $\Phi_{\xi}$ is a self-mapping onto $\mathcal{N}_{\rho}$.   
Consequently, for any $\xi_{k+1}\in\mathcal{N}_{\tau}$, $\Phi_{\xi_{k+1}}$ possesses a unique fixed point $\bar{z}^{k+1}$ in $\mathcal{N}_{\rho}$ by virtue of the \textit{contraction principle}. 
This statement is equivalent to $\bar{z}^{k+1}$ is a KKT point of $\textrm{P}(\xi_{k+1})$, i.e. $\bar{z}^{k+1}\in\Lambda(\xi_{k+1})$. Hence, $\Lambda_{k+1}\neq\emptyset$.

\noindent\textit{\underline{Step 1.4}.} Finally, we estimate $\norm{\bar{z}^{k+1} - \bar{z}^k}$. From the properties of $\Phi_{\xi}$ we have
\begin{equation}\label{eq:est_p1}
\norm{\bar{z}^{k+1} - z} \leq (1-c_0\gamma)^{-1}\norm{\Phi_{\xi_{k+1}}(z)-z}, ~~\forall z\in\mathcal{N}_{\rho}.  
\end{equation}
Using this inequality with $z = \bar{z}^k$ and noting that $\bar{z}^k = \Phi_{\xi_k}(\bar{z}^k)$, we have
\begin{equation}\label{eq:est_p1}
\norm{\bar{z}^{k+1} - \bar{z}^k} \leq (1-c_0\gamma)^{-1}\norm{\Phi_{\xi_{k+1}}(z)-\Phi_{\xi_k}(\bar{z}^k)}.  
\end{equation}
Since $\norm{r_k(\bar{z}^k;\xi_k) - r_k(\bar{z}^k;\xi_{k+1})} = \norm{M(\xi_{k+1}-\xi_k)}$, applying again (A3), it follows from \eqref{eq:est_p1} that
\begin{equation}\label{eq:norm_zk}
\norm{\bar{z}^{k+1} - \bar{z^k}} \leq (1-c_0\gamma)^{-1}\gamma\norm{M(\xi_{k+1}-\xi_k)}. 
\end{equation}

%+ Part 2.
\noindent\textbf{Part 2: } Let us define the residual from $\hat\varphi(z;\bar{x}^{k},\xi_{k+1})$ to $\hat\varphi(z;x^k,\xi_{k+1})$ as:
\begin{equation}\label{eq:delta_0}
\delta(z;x^k, \xi_{k+1}) := \hat\varphi(z; \bar{x}^k, \xi_{k+1}) - \hat\varphi(z; x^k, \xi_{k+1}).
\end{equation} 
\vskip -0.15cm
\noindent\textit{\underline{Step 2.1}.} We first provide an estimation for $\norm{\delta(z;x^k, \xi_{k+1})}$. From \eqref{eq:delta_0}
we have
\begin{eqnarray}\label{eq:proof_eq12}
\delta(z;x^k,\xi_{k+1}) &&\!\!\!= \left[\hat\varphi(z; \bar{x}^{k}\!\!, \xi_{k+1}\!) -\varphi(\bar{z}^{k}\!; \xi_{k+1})\right]-
\left[\varphi(z;\xi_{k+1}\!) - \varphi(\bar{z}^{k}\!; \xi_{k+1})\right] \notag\\
&&\!\!\! - \left[\hat\varphi(z; x^k, \xi_{k+1}) - \varphi(z^k; \xi_{k+1})\right] + \left[\varphi(z; \xi_{k+1}) - \varphi(z^k; \xi_{k+1})\right]\nonumber\\
&&\!\!\!=\int_0^1B(z^k_t;x^k)(z -z^k)dt - \int_0^1B(\bar{z}^{k}_t; \bar{x}^{k})(z-\bar{z}^{k})dt\\
&&\!\!\!=\int_0^1\!\!\!\!\left[B(z^k_t;x^k)-B(\bar{z}^{k}_t;\bar{x}^{k})\right](z -z^k)dt\!-\! \int_0^1\!\!\! B(\bar{z}^{k}_t; \bar{x}^{k})(z^k-\bar{z}^{k})dt,\nonumber
\end{eqnarray}
where $z^k_t := z^k + t(z-z^k)$, $\bar{z}^{k}_t := \bar{z}^{k} + t(z-\bar{z}^{k})$ and $B$ is defined by \eqref{eq:B_mat}.
Using the definition of $\hat\varphi$ and the estimations of $E_g$ and $g'$ at \textit{Step 1.1}, it is easy to show that 
\begin{eqnarray}\label{eq:est_1a}
&&\norm{B(z^k_t;x^k)-B(\bar{z}^{k}_t;\bar{x}^{k})} \leq \left[\norm{E_g(z^{k}_t\!) \!-\! E_g(\bar{z}^{k})}_F^2 \!+\! 2\norm{g'(x^k_t)-g'(x^k)}_F^2\right]^{1/2}\nonumber\\
[-1.5ex]\\[-1.5ex]
&& + \left[\norm{E_g(\bar{z}^{k}_t\!) \!-\! E_g(\bar{z}^{k})}_F^2 \!+\! 2\norm{g'(\bar{x}^k_t)-g'(\bar{x}^k)}_F^2\right]^{1/2}  \leq
2\sqrt{3}\varepsilon.\nonumber 
\end{eqnarray}
Similar to \eqref{eq:B_mat_est}, the quantity $B(\bar{z}^{k}_t; \bar{x}^{k})$ is estimated by
\begin{equation}\label{eq:est_1b}
\norm{B(\bar{z}^{k}_t; \bar{x}^{k})} \leq \kappa + \sqrt{3}\varepsilon.  
\end{equation}
Substituting \eqref{eq:est_1a} and \eqref{eq:est_1b} into \eqref{eq:proof_eq12}, we obtain an estimation for $\norm{\delta(z;x^k,
\xi_{k+1})}$ as
\begin{equation}\label{eq:delta_est}
\norm{\delta(z;x^k, \xi_{k+1})} \leq (\kappa + \sqrt{3}\varepsilon)\norm{z^k-\bar{z}^{k}} + 2\sqrt{3}\varepsilon\norm{z-z^k}. 
\end{equation}

\vskip -0.15cm
\noindent\textit{\underline{Step 2.2}.} We finally prove the inequality \eqref{eq:contr_est}.
Suppose that $z^{k+1}$ is a KKT point of \ref{eq:subprob_x}, we have $0\in\hat\varphi(z^{k+1};x^k,\xi_{k+1}) + N_K(z^{k+1})$.
This inclusion implies $\delta(z^{k+1};x^k, \xi_{k+1})\in \hat\varphi(z^{k+1};\bar{x}^{k},\xi_{k+1}) +
N_K(z^{k+1}) \equiv L(z^{k+1};\xi_{k+1})$ by the definition \eqref{eq:delta_0} of $\delta(z^{k+1};x^k, \xi_{k+1})$.
On the other hand, since $0\in\hat\varphi(\bar{z}^{k};\bar{x}^{k},\xi_{k}) + N_K(\bar{z}^{k})$, which is equivalent to
$\delta_1 := M(\xi_{k+1}-\xi_{k}) \in L(\bar{z}^k;\xi_{k+1})$, applying (A3) we get
\begin{eqnarray*}\label{eq:est_3}
\norm{z^{k+1}-\bar{z}^{k}} &&\leq\gamma\norm{\delta(z^{k+1};x^k, \xi_{k+1})-\delta_1} \nonumber\\
[-1.5ex]\\[-1.5ex]
&&\leq \gamma\norm{\delta(z^{k+1};x^k, \xi_{k+1})} + \gamma\norm{M(\xi_{k+1}-\xi_k)}.\nonumber 
\end{eqnarray*}
Combining this inequality and \eqref{eq:delta_est} with $z=z^{k+1}$ to obtain
\begin{eqnarray}\label{eq:est_4a}
\norm{z^{k\!+\!1}\!\!\!-\bar{z}^{k}} \leq \gamma(\kappa \!+\! \sqrt{3}\varepsilon)\norm{z^k \!-\! \bar{z}^k} \!+\!
2\sqrt{3}\gamma\varepsilon\norm{z^{k\!+\!1}\!\!\! -\! z^k} \!+\! \gamma\norm{M(\xi_{k+1}\! -\! \xi_k)}.
\end{eqnarray}
Using the triangular inequality, after a simple arrangement, \eqref{eq:est_4a} implies 
\begin{eqnarray}\label{eq:est_4}
\norm{z^{k+1}-\bar{z}^{k+1}} &&\leq \frac{\gamma(\kappa+3\sqrt{3}\varepsilon)}{1-2\sqrt{3}\gamma\varepsilon}\norm{z^k-\bar{z}^k} +
\frac{1+2\sqrt{3}\gamma\varepsilon}{1-2\sqrt{3}\gamma\varepsilon}\norm{\bar{z}^{k+1}-\bar{z}^k}\nonumber\\
[-1.5ex]\\[-1.5ex]
&& + \frac{\gamma}{1-2\sqrt{3}\gamma\varepsilon}\norm{M(\xi_{k+1}-\xi_k)}.\nonumber 
\end{eqnarray}
Now, let us define $\omega_k := \frac{\gamma(\kappa+3\sqrt{3}\varepsilon)}{1-2\sqrt{3}\gamma\varepsilon}$, $c_k :=
\frac{\gamma}{1-2\sqrt{3}\gamma\varepsilon}\left[\frac{2\sqrt{3}\gamma\varepsilon+1}{1-c_0\gamma} + 1\right]$.
By the choice of $\varepsilon$ at \textit{Step 1.1}, we can easily check that $\omega_k\in (0,1)$ and $c_k > 0$. Substituting \eqref{eq:norm_zk} into \eqref{eq:est_4} and using the definitions of $\omega_k$ and $c_k$, we obtain
\begin{equation*}
\norm{z^{k+1} - \bar{z}^{k+1}} \leq \omega_k\norm{z^k - \bar{z}^k} + c_k\norm{M(\xi_{k+1}-\xi_k)},
\end{equation*}
which proves \eqref{eq:contr_est}. The theorem is proved.
\eoproof
\end{proof}
%+ End of proof.

If $\Gamma \equiv \{\xi\}$ then the RTSCP method collapses to the full-step SCP
method described in Section \ref{se:intro}. Without loss of generality, we can assume that $\xi_k=0$ for all $k\geq 0$.  
The following corollary immediately follows from Theorem \ref{th:contract}.
  
%+ Corollary 2.1
\begin{corollary}\label{co:scp_conver}
\vskip -0.2cm
Suppose that $\set{z^j}_{j\geq 1}$ is the sequence of the KKT points of $\textrm{P}_{\textrm{cvx}}(x^{j-1};0)$ generated by the SCP method described in Section \ref{se:intro}  and that the assumptions of
Theorem \ref{th:contract} hold for $\xi_k=0$. Then 
\begin{equation}\label{eq.2.2.30}
\norm{z^{j+1} - \bar{z}} \leq \omega\norm{z^j - \bar{z}}, ~~\forall j\geq 0,
\end{equation}
where $\omega \in (0,1)$ is the contraction factor.
Consequently, this sequence converges linearly to a KKT point $\bar{z}$ of $\textrm{P}(0)$.
\end{corollary}

%%%%%%%%%%%%%%%%%%%%%%%%%%%%%%%%%%%%%%%%%%%%%%%%%%%%%%%%%%%%%%%%%%%%%%%%%%%%%%%
\section{Numerical example: control of an underactuated hovercraft}\label{se:example}
\vskip -0.3cm
In this section we apply RTSCP to the control of an underactuated hovercraft. We use the same model as in \cite{Seguchi2003}, which is characterized by the following differential equations: 
\begin{equation}\label{eq:hovercraft}
\begin{cases}
 m\ddot y_1(t) = (u_1(t) + u_2(t))\cos(\theta), \\
 m\ddot y_2(t) = (u_1(t) + u_2(t))\sin(\theta), \\
 I\ddot \theta(t) = r(u_1(t) - u_2(t)),
\end{cases}
\end{equation}
where $y(t)=(y_1(t),y_2(t))^T$ is the coordinate of the center of mass of the hovercraft (see Fig. \ref{fig:1}); $\theta(t)$ represents the
direction of the hovercraft; $u_1(t)$ and $u_2(t)$ are the fan thrusts; $m$ and $I$ are the mass and moment of inertia of the hovercraft, respectively; and $r$ is the distance between the central axis of the hovercraft and the fans.

%+ Figure 1: RC Hovercraft and its model
\setcounter{figure}{1}
\begin{figure}
\centering
\centerline{\includegraphics[angle=0,width=5.5cm, height=2.5cm]{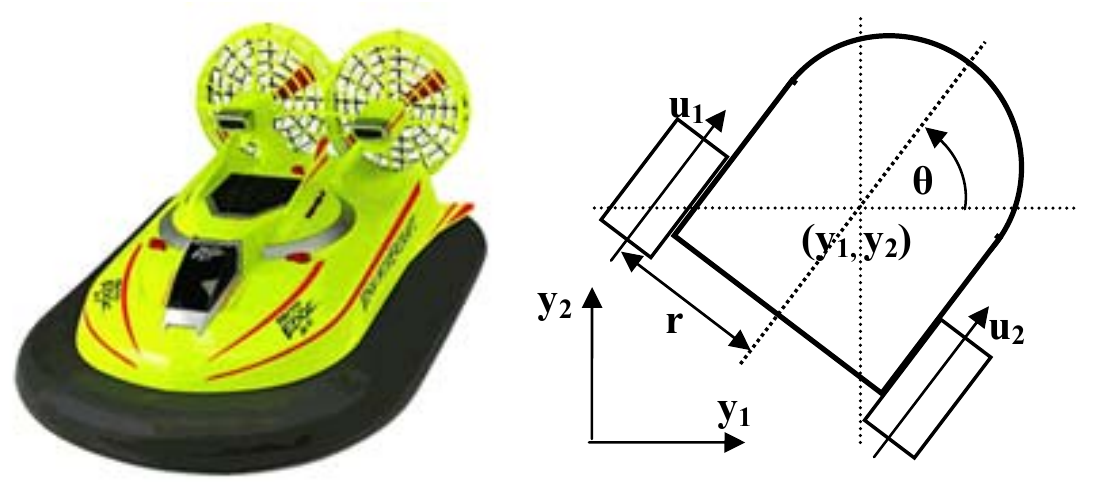}}
\caption{RC hovercraft and its model \cite{Seguchi2003}.}
\label{fig:1}       % Give a unique label
\vskip -0.4cm
\end{figure}

The problem considered is to drive the hovercraft from its initial position to the final parking position corresponding
to the origin of the state space while respecting the constraints
\begin{equation}\label{eq:constr}
\underline{u} \leq u_1(t) \leq \bar{u}, ~~ \underline{u} \leq u_2(t) \leq \bar{u}, ~~ \underline{y}_1 \leq y_1(t)\leq \bar{y}_1, ~~
\underline{y}_2 \leq y_2(t)\leq \bar{y}_2.
\end{equation}
To formulate this problem so that we can use the proposed method, we discretize the dynamics of the system using the Euler discretization scheme.
After introducing a new state variable $\xi := (y_1, y_2, \theta, \dot{y}_1, \dot{y_2}, \dot{\theta})^T$ and a control variable $u := (u_1,
u_2)^T$, we can formulate the following optimal control problem:
\begin{equation}\label{eq:ocp}
\begin{split}
\min_{\substack{\xi_0,\dots,\xi_N \\u_0,\dots,u_{N-1}}}~~   & \sum_{n=0}^{N-1} \left[\norm{\xi_n}_Q^2 + \norm{u_n}_R^2\right] + \norm{\xi_N}_S^2 \\
    \text{s.t.}~~~~~~~~                                     & \xi_0=\bar\xi, \\
                                                            & \xi_{n+1} = \phi(\xi_n,u_n) \quad \forall n=0,\dots,N-1,\\
                                                            & (\xi_0,\dots,\xi_N,u_0,\dots,u_{N-1}) \in \tilde\Omega,
\end{split}
\end{equation}
where $\phi(\cdot,\cdot)$ represents the discretized dynamics and the constraint set $\tilde\Omega$ can be easily deduced from \eqref{eq:constr}.
By introducing a slack variable $s$ and using the convex constraint:
\begin{equation}
s\geq\sum_{n=0}^{N-1} \left[\norm{\xi_n}_Q^2 + \norm{u_n}_R^2\right] + \norm{\xi_N}_S^2,
\end{equation}
we can transform \eqref{eq:ocp} into $\textrm{P}(\bar\xi)$ of a variable $x := (s, \xi_0^T,\dots,\xi_N^T, u_0^T,\dots, u_{N-1}^T)^T$ and
the objective function $c^Tx = s$. Note that $\bar{\xi}$ is an online parameter. It plays the role of $\xi_k$ in the RTSCP algorithm
along the moving horizon (see Section \ref{se:RTSCP}).

We implemented the RTSCP algorithm using a primal-dual interior point method for solving the convex subproblem
$\textrm{P}_{\textrm{cvx}}(x^{k-1};\xi_{k})$. We performed a simulation using the same data as in \cite{Seguchi2003}:
$m=0.974\text{kg}$, $I = 0.0125\text{kg}\cdot \text{m}^2$, $r = 0.0485\text{m}$, $\underline{u} = -0.121\text{N}$,  $\bar{u} =
0.342\text{N}, \underline{y}_1 = \underline{y}_2 = -2\text{m}$,
$\bar{y}_1 = \bar{y}_2 = 2\text{m}$, $Q=\text{diag}(5, 10, 0.1, 1, 1, 0.01)$, $S=\text{diag}(5, 15, 0.05, 1, 1, 0.01)$, $R=\text{diag}(0.01,
0.01)$ and the initial condition $\xi_0 = \xi(0) = (-0.38, 0.30, 0.052, 0.0092, -0.0053, 0.002)^T $.

Figure \ref{fig:2} shows the results of the simulation where a sampling time of $\Delta t = 0.05s$ and $N=15$ are used.
The stopping condition used for the simulation is $\norm{y(t)}\leq 0.01$.

%+ Figure 2. Controller and motion in two dimension
\begin{figure}
\vskip -0.2cm
\centering
\centerline{\includegraphics[angle=0,width=8.5cm, height=3.8cm]{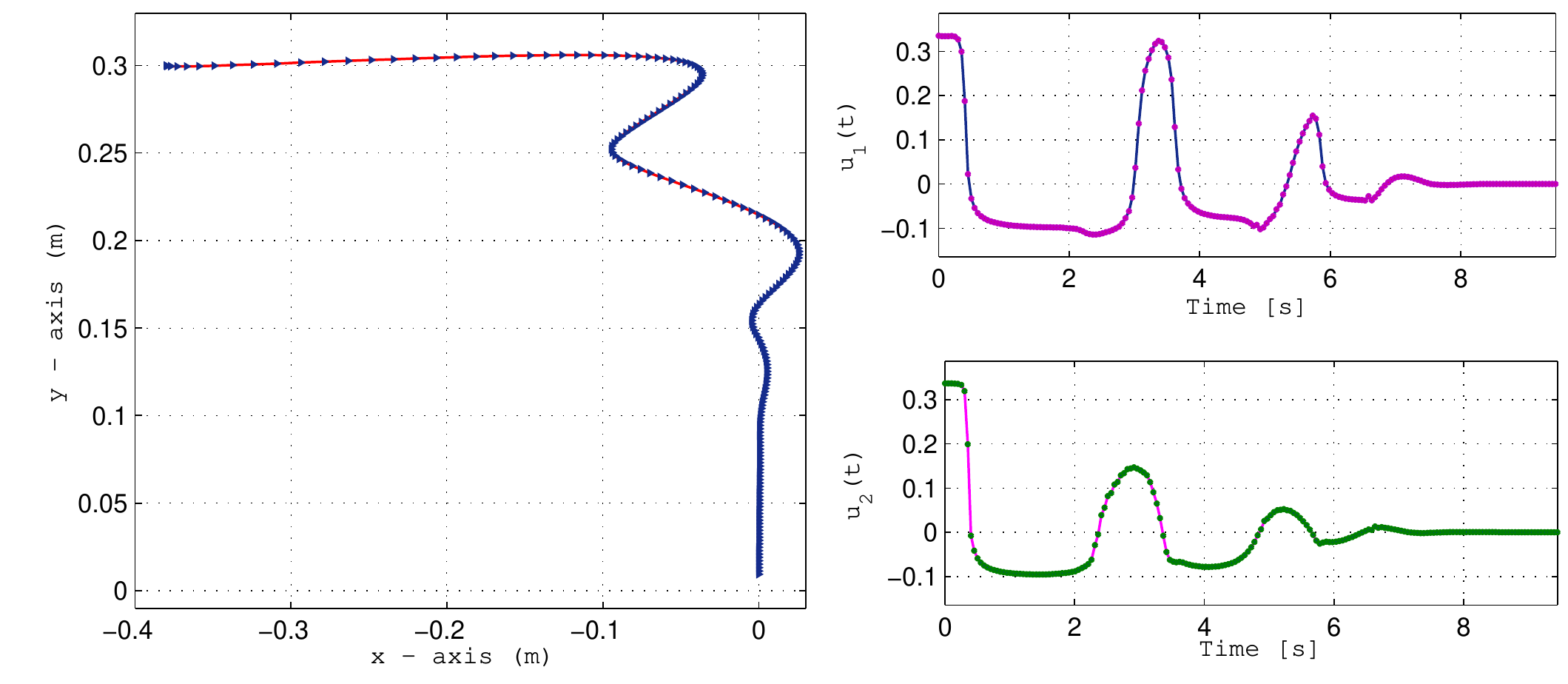}}
\caption{\footnotesize Trajectory of the hovercraft after $t = 9.5$s (left) and control input profile (right).}
\label{fig:2}  % Give a unique label
\end{figure}

%%%%%%%%%%%%%%%%%%%%%%%%%%%%%%%%%%%%%%%%%%%%%%%%%%%%%%%%%%%%%%%%%%%%%%%%%%%%%%%%
%+ Acknowledgments.
\vskip 0.2cm

\noindent{\footnotesize\textbf{Acknowledgments.}}
The authors would like to thank the anonymous referees for their comments and suggestions that helped to improve the paper.

{\scriptsize
This research was supported by Research Council KUL: CoE EF/05/006 Optimization in Engineering(OPTEC), GOA AMBioRICS, IOF-SCORES4CHEM, several PhD/postdoc \& fellow grants; the Flemish Government via FWO: PhD/postdoc grants, projects G.0452.04, G.0499.04, G.0211.05, G.0226.06, G.0321.06, G.0302.07, G.0320.08 (convex MPC), G.0558.08 (Robust MHE), G.0557.08, G.0588.09, research communities (ICCoS, ANMMM, MLDM) and via IWT: PhD Grants, McKnow-E, Eureka-Flite+EU: ERNSI; FP7-HD-MPC (Collaborative Project STREP-grantnr. 223854), Contract Research: AMINAL, and Helmholtz Gemeinschaft: viCERP; Austria: ACCM, and the Belgian Federal Science Policy Office: IUAP P6/04 (DYSCO, Dynamical systems, control and optimization, 2007-2011).
}

%%%%%%%%%%%%%%%%%%%%%%%%%%%%%%%%%%%%%%%%%%%%%%%%%%%%%%%%%%%%%%%%%%%%%%%%%%%%%%%%
%+ References
%%%%%%%%%%%%%%%%%%%%%%%%%%%%%%%%%%%%%%%%%%%%%%%%%%%%%%%%%%%%%%%%%%%%%%%%%%%%%%%%

%\printindex
\end{document}